%% file: main.tex
\documentclass[12pt]{article}

\input{preamble}

\begin{document}

\title{Graded Ehrhart theory and toric geometry}
\author{Ian Cavey}
\date{\today}
\maketitle

\begin{abstract}
    We give two new constructions of the \textit{harmonic algebra} of a lattice polytope $P$, a bigraded algebra whose character is the $q$-Ehrhart series of $P$ defined by Reiner and Rhoades \cite{RR24}. First, we show that the harmonic algebra is the associated graded algebra of the semigroup algebra of $P$ with respect to a certain natural filtration, clarifying it's relationship with the more classical semigroup algebra. We then give a geometric interpretation of the harmonic algebra as a quotient of the ring of global sections of a certain family of line bundles on the blowup of the toric variety associated to $P$ at a generic point. Using this connection to toric geometry we resolve one the main conjectures of Reiner and Rhoades \cite{RR24} by showing that the harmonic algebra is \textit{not} finitely generated in general.
\end{abstract}

\section{Introduction}

The \textit{Ehrhart series} of a lattice polytope $P\subseteq \R^n$ is the generating function for the sequence of integer point counts in the dilations $mP = \{ (mp_1,\dots,mp_n) \,|\, (p_1,\dots,p_n) \in P\}$,
\begin{equation}
    E_P(t) = \sum_{m\geq 0} \left| (mP) \cap \Z^n \right| \, t^m. 
\end{equation} 
Although the Ehrhart series is defined as a formal power series, it can always be expressed as a rational function of $t$. A powerful algebraic tool for studying Ehrhart series is the \textit{semigroup algebra} of $P$,
\begin{equation}\label{semigroupalgebra}
    A_P = \bigoplus_{\substack{(p_1,\dots,p_n)\in (mP)\cap \Z^n\\ m\geq 0}} \C \, x_1^{p_1}\cdots x_n^{p_n} z^m \subseteq \C[x_1^{\pm 1},\dots, x_n^{\pm 1}][z],
\end{equation} 
a graded algebra whose Hilbert series is $E_P(t)$ by construction. This algebra can be interpreted geometrically as the coordinate ring of the toric variety $X_P$ associated to $P$, providing a fruitful connection between algebraic geometry and polytopal combinatorics. 

In this paper, we study $q$-analogues of these objects recently introduced by Reiner and Rhoades \cite{RR24}. The \textit{$q$-Ehrhart series} of $P$, denoted $E_P(t,q)$, is a refinement of the classical Ehrhart series of $P$. The coefficient of $t^m$ in the $q$-Ehrhart series is a polynomial in $q$ with nonnegative integer coefficients that specializes at $q=1$ to the number of integer points in $mP$. The precise definition of these coefficients in recalled in Section \ref{sec:background}.

By analogy with the semigroup graded algebra for Ehrhart series, Reiner and Rhoades constructed a bigraded algebra $\H_P$, called the \textit{harmonic algebra} of $P$, whose bigraded Hilbert series is $E_P(t,q)$. The definition of the harmonic algebra, recalled in Section \ref{sec:background}, is somewhat involved as it requires dualizing with respect to the action of one polynomial ring on another via partial derivatives. In particular, as noted by Reiner and Rhoades, it is not clear from the definition that $\H_P$ is closed under multiplication. Nonetheless, the harmonic algebra can be explicitly presented for small examples and for some special classes of polytopes (\cite{RR24} Proposition 5.7). Based on this evidence, Reiner and Rhoades conjectured the following.

\begin{conjecture}
    \label{conj}(\cite{RR24} Conjecture 5.5 (i))
    For any lattice polytope $P$, the harmonic algebra $\H_P$ is finitely generated. In particular $E_P(t,q)$ is a rational function of $t$ and $q$.
\end{conjecture}

In this paper, we resolve the algebraic part of Conjecture \ref{conj} in the negative, showing that there exist lattice triangles $P$ for which $\H_P$ is not finitely generated (see Example \ref{ex:nonfg}). The weaker conjecture, the rationality of the $q$-Ehrhart series $E_P(t,q)$, remains open. 

Our results on the harmonic algebra are based on the following simple description of $\H_P$ established in Section \ref{sec:harmonic}. Recall that the order of vanishing of a Laurent polynomial $f(x_1,\dots,x_n)$ at a point $p$ is the smallest degree $d$ such that some partial derivative of $f$ of order $d$ is nonzero at $p$.

\begin{theorem}\label{thm:filtration}
    For any lattice polytope $P$, the harmonic algebra $\H_P$ is isomorphic as a bigraded algebra to the associated graded algebra of $A_P$ with respect to the filtration by order of vanishing at $e=(1,\dots,1)\in \R^n$.
\end{theorem}

\begin{example}\label{ex:harmonics}
    Let $P\subseteq \R^2$ be the convex hull of the points $(0,0),(2,1)$, and $(1,2)$.
    \begin{figure}[h]
        \centering
        \begin{tikzpicture}[scale=1.7]
            \filldraw[fill opacity = .2,thick] (0,0) -- (2,1) -- (1,2) -- cycle;
             \node[label=right:{$(2,1)$}] at (2,1) (nodeA) {};
             \node[label=right:{$(1,2)$}] at (1,2) (nodeB) {};
             \node[label=right:{$(1,1)$}] at (1,1) (nodeB) {};
             \node[label=left:{$(0,0)$}] at (.02,.1) (nodeC) {};
            \filldraw (0,0) circle (1pt);
            \filldraw (1,1) circle (1pt);
            \filldraw (2,1) circle (1pt);
            \filldraw (1,2) circle (1pt);
        \end{tikzpicture}
        \caption{The lattice triangle $P$}
        \label{fig:triangle}
    \end{figure}
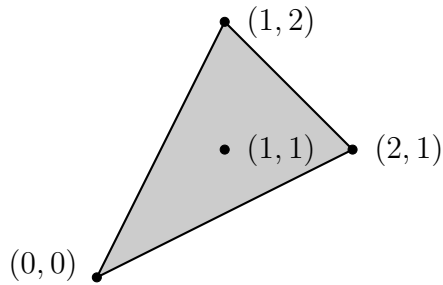
    The triangle $P$ contains $4$ integer points so we have $E_P(t)=1+4t+\cdots$. Using the definition (\ref{semigroupalgebra}), the degree-one piece of the semigroup algebra is the linear span $(A_P)_1 \simeq \langle 1, xy,x^2y, xy^2 \rangle$, where we write $x=x_1$ and $y=x_2$, and drop the grading parameter $z$. Let us write $F_{1,d}$ for the subset of $(A_P)_1$ consisting of polynomials that vanish at $e$ with order at least $d$. The filtration $(A_P)_1 = F_{1,0} \supseteq F_{1,1}\supseteq F_{1,2}\supseteq F_{1,3}\supseteq \cdots$ is given by 
    \[ (A_P)_1 = \langle 1, xy,x^2y, xy^2 \rangle \supseteq \langle xy-1,x^2y-1,xy^2-1 \rangle \supseteq \langle x^2 y + x y^2 -3xy+1 \rangle \supseteq \langle 0 \rangle \supseteq \cdots\]
    Theorem \ref{thm:filtration} says there is an identification $(\H)_{m,d}\simeq F_{m,d}/F_{m,d+1}$ that preserves the multiplicative structure between the graded pieces. In particular, the coefficient on $t$ in $E_P(t,q) = 1 + (1+2q+1)t+\cdots$ is the generating function for the differences in dimension of consecutive filtered pieces of $(A_P)_1$.
\end{example}

Next, we connect the harmonic algebra to toric geometry. Let us briefly recall the classical setup of toric geometry (see, for example, \cite{F93}). Let $X_P$ be the toric variety associated to the lattice polytope $P$. The variety $X_P$ contains a dense open subset that is identified with the algebraic torus $T = (\C^*)^n$, on which the regular functions are the Laurent polynomials $\C[x_1^{\pm1},\dots,x_n^{\pm1}]$. The toric variety $X_P$ is equipped with a line bundle $\Oo(H)$, and the global sections of any power $\Oo(mH)$ are identified, via restriction to $T$, with Laurent polynomials in $x_1,\dots,x_n$ supported on $mP$. This gives the well-known identification
\begin{equation}
    A_P \simeq \bigoplus_{m\geq 0} H^0(X_P,\Oo(mH)).
\end{equation} 
Now let $\Bl_e X_P$ denote the blowup of $X_P$ at the identity point of the torus $e=(1,\dots,1)\in T\subseteq X_P$, with exceptional divisor $E\subseteq \Bl_e X_P$. Since $e\in X_P$ is not $T$-fixed, $\Bl_e X_P$ is not, in general, a toric variety. Consider the bigraded section ring
\begin{equation} 
    R_P := \bigoplus_{m,d \in \Z} H^0(\Bl_e X_P, \Oo(mH-d E)),
\end{equation}
where we identify $\Oo(mH)$ with its pullback to $\Bl_e X_P$. The line bundle $\Oo(E)$ has a canonical section $s$, whose zero locus is $E$, which is a regular element of $R_P$. 

The following result, proved in Section \ref{sec:toric}, is a geometric rephrasing of Theorem \ref{thm:filtration}.

\begin{theorem}\label{thm:toric}
    For any lattice polytope $P$ we have $\H_P\simeq R_P/(s)$ as bigraded algberas.
\end{theorem}

Theorem \ref{thm:toric} connects the harmonic algebra and $q$-Ehrhart theory to a large body of algebraic and geometric results about finite generation of section rings of blowups of toric varieties. The algebraic properties of these rings are subtle already for lattice triangles $P$ for which the toric variety $X_P$ is a weighted projective plane. For example, Cutkosky \cite{C91} studied the finite generation of sections rings for blowups of weighted projective planes described algebraically as symbolic algebras of monomial prime ideals. Specific examples of weighted projective planes, and their corresponding lattice triangles, whose blowups have non-finitely generated section rings have been given by Goto, Nishida, and Watanabe \cite{GNW94} as well as Gonz\'alez and Karu \cite{GK16}.

By Theorem \ref{thm:toric}, the harmonic algebra $\H_P$ is finitely generated if and only if the section ring $R_P$ is finitely generated. This connection, combined with the previous work described above, therefore resolves the algebraic portion of Conjecture \ref{conj} in the negative.

\begin{corollary}
    There exist lattice triangles $P$ such that the harmonic algebra $\H_P$ is not finitely generated.
\end{corollary}

In Example \ref{ex:nonfg}, adapted from an example studied by Gonz\'alez and Karu \cite{GK16}, we show that the harmonic algebra of the triangle with vertices $(0,0), (7,56),$ and $(-45,30)$ is not finitely generated. Our argument for non-finite generation depends on the geometric results of Gonzalez and Karu. It would be interesting to give an elementary proof of the non-finite generation of the harmonic algebra for this example.

A natural next problem is to determine for which polytopes the harmonic algebra $\H_P$ is finitely generated. On the geometric side, this is an ongoing area of research largely focused on the case when $P$ is a lattice triangle. Cutkosky \cite{C91} gave a geometric criterion (depending on of the existence of certain curves on $\Bl_e X_P$) for the section ring $R_P$ to be finitely generated. It is not clear, however, how the existence of such curves depends on the combinatorics of the triangle. In addition to the examples referred to above, a systematic study of this question was recently initiated by Anaya, Gonz\'alez, and Karu \cite{AGK25}.

We hope that this connection to toric geometry will prove useful for resolving other conjectures about the harmonic algebra, and provide a new perspective on these toric geometry problems as well. \\

\textbf{Acknowledgments:}
I thank Victor Reiner for the engaging discussion about $q$-Ehrhart theory, and both Victor Reiner and Brendon Rhoades for their comments and suggestions on an earlier draft of this work.

\section{Background}\label{sec:background}

We first review the constructions and basic facts about $q$-Ehrhart series and the harmonic algebra from \cite{RR24}. We restrict to work over $\C$, rather than an arbitrary field as in \cite{RR24}, which will be needed for our connection to toric geometry.

Given a finite set $Z\subseteq \C^n$, the ideal of $Z$ is 
\begin{equation}
    I(Z) = \{ f \, | \, f(p)=0 \text{ for all }p\in Z\}\subseteq \C[x_1,\dots,x_n].
\end{equation}
Geometrically, the quotient $\C[x_1,\dots,x_n]/I(Z)$ is the coordinate ring of the reduced union of points $Z$. Now consider the ideal $\gr I(Z) = \{ f_d \, | \, f = f_0+f_1+\cdots +f_d\in I(Z)\}$, where $f_i$ is the homogeneous degree $i$ piece of $f$. The quotient $\C[x_1,\dots,x_n]/\gr I(Z)$ is the coordinate ring of the scheme obtained by degenerating the set $Z$ to a reduced scheme supported at the origin. Unlike $I(Z)$, $\gr I(Z)$ is a homogeneous ideal and so $\C[x_1,\dots,x_n]/\gr I(Z)$ is graded by total degree in $x_1,\dots,x_n$.

\begin{definition}
    The $q$-Ehrhart series of a lattice polytope $P$ is the formal power series
    \[ E_P(t,q) = \sum_{m,d\geq 0} \dim_\C \big( \C[x_1,\dots,x_n]/\gr I((mP)\cap \Z^n) \big)_d\, t^m q^d. \]
\end{definition}

One can show using Gr\"obner bases that $\C[x_1,\dots,x_n]/I(Z)$ and $\C[x_1,\dots,x_n]/\gr I(Z)$ have the same dimension as $\C$-vector spaces. It follows that $E_P(t,q)$ is a refinement of $E_P(t)$, in the sense that $E_P(t,1) = E_P(t)$.

\begin{definition}
    The \textit{harmonic space} of a finite set $Z\in \C^n$ is the set
    \[ V_Z = \left\{ g\in \C[x_1,\dots,x_n] \, \bigg| \, f\left( \frac{\partial}{\partial x_1},\dots, \frac{\partial}{\partial x_n}\right) g(x_1,\dots,x_n) = 0 \text{ for all } f\in \gr I(Z) \right\}. \]
\end{definition}

The harmonic space is linear subspace $V_Z\subseteq \C[x_1,\dots,x_n]$ of dimension $|Z|$ and is graded by total degree in $x_1,\dots,x_n$.

\begin{definition}
    The \textit{harmonic algebra} of $P$ is defined by
    \[ \H_P = \bigoplus_{m\geq 0} V_{(mP)\cap \Z^n}, \]
    with the bigrading $(\H_P)_{m,d} = (V_{(mP)\cap \Z^n})_d$. 
\end{definition}
By Proposition 5.4 of \cite{RR24}, $\H_P$ forms a bigraded algebra with respect to the usual multiplication of polynomials, and its bigraded Hilbert series is the $q$-Ehrhart series of $P$:
\begin{equation}
    E_P(t,q) = \sum_{m,d\geq 0} \dim_C (\H_P)_{m,d} \, t^m q^d.
\end{equation}

\section{The harmonic algebra via the semigroup algebra}\label{sec:harmonic}

We begin with an alternative characterization of the harmonic space $V_Z$ for any finite set $Z\subseteq \C^n$ that may be of interest in its own right. In the following, for any point $\a = (a_1,\dots,a_n)\in \C^n$ we write $e^{\a\cdot\x} = e^{a_1x_1+\cdots+a_nx_n}$ and $(1+\x)^{\a} = (1+x_1)^{a_1}\cdots(1+x_n)^{a_n}$, both considered as formal power series in $x_1,\dots,x_n$ after expanding at $(0,\dots,0)$.

\begin{proposition}\label{prop:harmonics}
    For any finite set $Z\in \C^n$, the harmonic space $V_Z$ is the set of lowest degree nonzero homogeneous parts of formal power series in the linear span $\langle (1+\x)^{\a} \, | \, \a\in Z \rangle$.
\end{proposition}

\begin{proof}
    Let $W_Z$ be the set of formal power series $g(x_1,\dots,x_n)$ such that for all $f\in I(Z)$
    \[ f\left( \frac{\partial}{\partial x_1},\dots, \frac{\partial}{\partial x_n}\right) g(x_1,\dots,x_n) = 0.\]
    
    By Corollary 4.4 of \cite{RR24}, $V_Z$ is the set of lowest degree homogeneous parts of elements of $W_Z$. By Lemma 4.6 of \cite{RR24}, $W_Z$ is the linear span $\langle e^{\a\cdot\x} \, | \, \a\in Z \rangle$ where we consider $e^{\a\cdot\x} = e^{a_1x_1+\cdots+a_nx_n}$ as a formal power series in $x_1,\dots,x_n$. The proposition therefore follows from the following observation.
\end{proof}

\begin{lemma}
    For any finite set $Z\in \C^n$ and set of constants $\{ c_{\a} \,|\, a\in Z\}$, the lowest degree nonzero homogeneous parts of the series 
    $\sum_{\a\in Z} c_{\a}e^{\a\cdot\x}$ and $\sum_{\a\in Z} c_{\a}(1+\x)^{\a}$ coincide.
\end{lemma}

\begin{proof}
    The $x_1^{d_1}\cdots x_n^{d_n}$ coefficients on the series $\sum_{\a\in Z} c_{\a}e^{\a\cdot\x}$ and $\sum_{\a\in Z} c_{\a}(1+\x)^{\a}$ are given by
    \[ E(d_1,\dots,d_n) := \sum_{\mathbf{a}\in Z} c_{\mathbf{a}} \frac{a_1^{d_1}\cdots a_n^{d_n}}{d_1 ! \cdots d_n !}  \hspace{1cm} \text{and}\hspace{1cm} F(d_1,\dots,d_n) := \sum_{\mathbf{a}\in Z} c_{\mathbf{a}} {a_1 \choose d_1} \cdots {a_n \choose d_n} \]
    respectively. Consider the monomial expansion of ${y_1 \choose d_1} \cdots {y_n \choose d_n}$ as a polynomial in $y_1,\dots,y_n$, which has unique highest degree term $\frac{y_1^{d_1}\cdots y_n^{d_n}}{d_1 ! \cdots d_n !}$. Applying this expansion to each term in the sum for $F(d_1,\dots,d_n)$ gives a relation $F(d_1,\dots,d_n) = E(d_1,\dots,d_n)+ \cdots$ where the remaining terms are linear combinations of lower degree coefficients of $\sum_{\mathbf{a}\in Z} c_{\mathbf{a}} e^{\mathbf{a}\cdot \mathbf{x}}$. This implies that the nonzero coefficients on monomials of minimum total degree for the two series coincide, as claimed.
\end{proof}

\begin{definition}
    For a lattice polytope $P$, let $F_{m,d}\subseteq (A_P)_m$ denote the set of elements represented by Laurent polynomials $f(x_1,\dots,x_n)z^m$, as in (\ref{semigroupalgebra}), such that $f(x_1,\dots,x_n)$ vanishes to order at least $d$ at the point $e=(1,\dots,1)$. 
\end{definition}

These subsets equip the semigroup algebra with a descending filtration $(A_P)_m = F_{m,0}\supseteq F_{m,1}\supseteq \cdots$. These filtered pieces are compatible in the sense that $F_{m,d}\cdot F_{m',d'}\subseteq F_{m+m',d+d'}$ for all $m,m',d,$ and $d'$, so we may define the associated graded algebra of $A_P$ as the bigraded algebra 
\begin{equation}\label{assgraded}
    \gr A_P = \bigoplus_{m,d\geq 0} F_{m,d}/F_{m,d+1}.
\end{equation}

Theorem \ref{thm:filtration} can now be quickly derived from Proposition \ref{prop:harmonics}.

\begin{proof}[Proof of Theorem \ref{thm:filtration}]
    An element $f(x_1,\dots,x_n)z^m\in (A_P)_m$ is in the filtered piece $F_{m,d}$ if and only if the monomial expansion of $f(x_1+1,\dots,x_n+1)$ contains only monomials of total degree at least $d$. We can therefore consider the quotient map $F_{m,d}\to F_{m,d}/F_{m,d+1}$ as sending such an element to the homogeneous degree $d$ part of $f(x_1+1,\dots,x_n+1)$. By Proposition \ref{prop:harmonics}, the image $F_{m,d}/F_{m,d+1}$ is exactly the degree $d$ graded piece of $V_{(mP)\cap \Z^n} = (\H_P)_{m,*}$. This identifies $\H_P$ with $\gr A_P$ as bigraded algebras, as desired.
\end{proof}

\section{The harmonic algebra via toric geometry}\label{sec:toric}

In this Section we prove Theorem \ref{thm:toric} which states that the harmonic algebra is isomorphic to the quotient of the section ring $R_P = \bigoplus_{m,d} H^0(\Bl_e X_P, \Oo(mH-d E))$ by the ideal generated by the canonical section $s\in \Oo(E)$. 

\begin{proof}[Proof of Theorem \ref{thm:toric}]
    For $d\geq 0$, multiplication by $s^d$ gives an injective map
    \[ H^0(\Bl_e X_P, \Oo(mH-d E))\xhookrightarrow{\cdot s^d} H^0(\Bl_e X_P, \Oo(mH)) \simeq H^0(X_P, \Oo(mH))\simeq (A_P)_m  \]
    whose image is exactly the set of sections that vanish to order at least $d$ at the point $e$, and
    \[ H^0(\Bl_e X_P, \Oo(mH))\xrightarrow{\cdot s^d} H^0(\Bl_e X_P, \Oo(mH+dE)) \]
    is an isomorphism. For all $m$ and $d$ we can therefore identify $(R_P)_{m,d} = H^0(\Bl_e X_P, \Oo(mH-d E))$ with $F_{m,d}\subseteq (A_P)_m$, the Laurent polynomials vanishing to order at least $d$ at $e$. In other words, we have the identification of bigraded rings
    \begin{equation}
        R_P \simeq \bigoplus_{m,d} F_{m,d}
    \end{equation} 
    
    Under this identification, the inclusions $F_{m,d+1}\subseteq F_{m,d}$ correspond multiplication by $s$ between the graded pieces of $R_P$. The quotient $R_P/(s)$ is therefore isomorphic to $\bigoplus_{m,d} F_{m,d}/F_{m,d+1}$, which is isomorphic to the harmonic algebra by Theorem \ref{thm:filtration}.
\end{proof}

    Finally, we analyze an example due to González and Karu, Example 1.3 of \cite{GK16}, that gives a non-finitely generated harmonic algebra.

\begin{example}\label{ex:nonfg} 
    Using the geometric terminology of González and Karu, the blowup of the weighted projective space $\P(15,26,7)$ at $e$ is not a Mori dream space, so the section ring $R_P$ is not finitely generated for any polytope $P$ such that $X_P\simeq \P(15,26,7)$. A minimal integral such triangle has vertices $(0,0), (7,56),$ and $(-45,30)$.
    
    We briefly describe González and Karu's method to show that $R_P$ is not finitely generated. For this, it is convenient to take $P$ to be a non-integral rational triangle. Let $H$ be the divisor on $\widetilde{X} = \mathrm{Bl}_e \P(15,26,7)$ corresponding to the rational triangle $P$ with vertices $(0,0),(2/15,6/15),(-6/7,4/7),$ and $E$ the exceptional divisor. This triangle $P$ is the smallest dilation of the integral triangle above that contains at least two lattice points: $(0,0)$ and $(0,1)$. This means $\Oo(H)$ has a two-dimensional space of sections spanned by the Laurent polynomials $1$ and $y$. The curve $C$ defined by $y-1=0$ on $\widetilde{X}$ has class $H-E$ and lies on the boundary of the effective cone of $\widetilde{X}$. 

    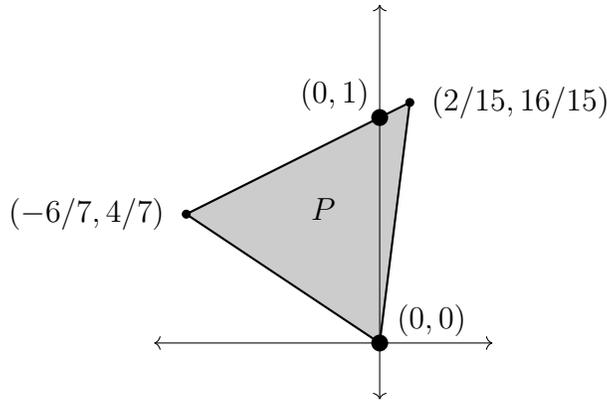
\begin{figure}[h]
        \centering
        \begin{tikzpicture}[scale = 3]
            \filldraw[fill opacity = .2,thick] (0,0) -- (-6/7,4/7) -- (2/15,16/15) -- cycle;
            \node[label = {$P$}] at (-1/4,.45) (nodeA) {};
             \node[label=left:{$(-6/7,4/7)$}] at (-6/7,4/7) (nodeB) {};
             \node[label=right:{$(2/15,16/15)$}] at (2/15,16/15) (nodeC) {};
             \node[label=left:{$(0,1)$}] at (.05,1.1) (nodeC) {};
             \node[label=right:{$(0,0)$}] at (-.02,.1) (nodeC) {};
            \draw[<->] (-1,0) -- (1/2,0);
            \draw[<->] (0,-1/4) -- (0,3/2);
            \filldraw (0,1) circle (1pt);
            \filldraw (0,0) circle (1pt);
            \filldraw (2/15,16/15) circle (.5pt);
            \filldraw (-6/7,4/7) circle (.5pt);
        \end{tikzpicture}
        \caption{A rational triangle $P$}
        \label{fig:triangle}
    \end{figure}
    
    As an intermediate step towards showing that $R_P$ is not finitely generated, González and Karu show that the nef divisor $D= H-\frac{104}{105}E$ contains $C$ in its stable base locus. Concretely, the geometric facts listed above mean that for any positive integer $m$:
    \begin{enumerate}
        \item The maximum order of vanishing at $e$ of a Laurent polynomial supported on $mP$ is exactly $m$, which is attained by $(y-1)^m$. 
        \item Any Laurent polynomial supported on $mP$ with order of vanishing at $e$ least $\frac{104}{105}m$ is divisible by $y-1$.
        \item For any $0\leq d<\frac{104}{105}m$ there exists an integer $k$ such that there is a Laurent polynomial supported on $kmP$ that is not divisible by $y-1$ and has order of vanishing $kd$ at $e$.
    \end{enumerate}
    Together, these facts imply that $\H_P$ cannot be finitely generated. Indeed, for any finitely generated subalgebra of $\H_P$, there is a constant $\alpha<104/105$ defined as the maximum value of $d/m$ among the generators of bidegree $(m,d)$ not divisible by $y-1$. Any element of the subalgebra of bidegree $(m,d)$ with $d/m>\alpha$ must be divisible by $y-1$. For any integers $d,m$ such that $\alpha <d/m<104/105$, property 3 above implies that the subalgebra does not contain all elements of $\H_P$ of bidegree $(km,kd)$ for all $k$.
\end{example}

\begingroup
\setstretch{1}
\bibliographystyle{plain}
\bibliography{refs.bib} 
\endgroup

\end{document}

%% file: preamble.tex
\usepackage[margin=1in]{geometry}

\usepackage{amsmath}
\usepackage{amsfonts}
\usepackage{amsthm}
\usepackage{mathrsfs}
\usepackage{amssymb}
\usepackage{tikz}
\usepackage{comment}
\usetikzlibrary{calc}
\usepackage{tikz-cd}
\usepackage{graphicx}
\usepackage{mathtools}
\usepackage{setspace}
\usepackage{blkarray}
\usepackage{stmaryrd}

\renewcommand{\a}{\mathbf{a}}
\newcommand{\x}{\mathbf{x}}
\renewcommand{\H}{\mathcal{H}}
\renewcommand{\P}{\mathbb{P}}
\DeclareMathOperator{\gr}{\mathsf{gr}}

\DeclareMathOperator{\Bl}{Bl}

\DeclareMathOperator{\Oo}{\mathcal{O}}

\newcommand{\R}{\mathbb R}

\newcommand{\Z}{\mathbb Z}
\newcommand{\C}{\mathbb C}

\newtheorem{theorem}{Theorem}[section]
\newtheorem{lemma}[theorem]{Lemma}
\newtheorem{corollary}[theorem]{Corollary}

\newtheorem{proposition}[theorem]{Proposition}
\newtheorem{conjecture}[theorem]{Conjecture}

\theoremstyle{definition}
\newtheorem{definition}[theorem]{Definition}
\newtheorem{example}[theorem]{Example}

\theoremstyle{remark}